\documentclass[11pt]{amsart}
\usepackage[a4paper,left=2.5cm,right=2.5cm]{geometry}
\usepackage[latin1]{inputenc}
\usepackage[english]{babel}
\usepackage{amsmath}
\usepackage{verbatim}
\usepackage{amssymb}
\usepackage{color}
\usepackage{hyperref}
\usepackage{amsthm}
\usepackage{tikz}
\usepackage{pgfplots}
\usepackage{tikz-3dplot}
\usepackage{float}
\usepackage{amsfonts} 
\usepackage{enumerate}
\usepackage{bbm}

\newtheorem{theorem}{Theorem}[section]

\newtheorem{lemma}[theorem]{Lemma}
\theoremstyle{definition}
\newtheorem{remark}{Remark}[section]
\theoremstyle{definition}
\newtheorem{example}{Example}[section]
\theoremstyle{definition}
\newtheorem{definition}{Definition}[section]
\begin{document}
\title{Steklov-Dirichlet spectrum: stability, optimization and continuity of eigenvalues}
\author[]{Marco Michetti}

\address[Marco Michetti]{
Institut Elie Cartan de Lorraine \\ CNRS UMR 7502 and Universit\'e de Lorraine \\
BP 70239
54506 Vandoeuvre-l\`es-Nancy, France}
\email{marco.michetti@univ-lorraine.fr}

\date{}

\maketitle
\begin{abstract}
In this paper we study the Steklov-Dirichlet eigenvalues $\lambda_k(\Omega,\Gamma_S)$, where $\Omega\subset \mathbb{R}^d$ is a domain and $\Gamma_S\subset \partial \Omega$ is the subset of the boundary in which we impose the Steklov conditions. After a first discussion about the regularity properties of the Steklov-Dirichlet eigenfunctions we obtain a stability result for the eigenvalues. We study the optimization problem under a measure constraint on the set $\Gamma_S$, we prove the existence of a minimizer and the non-existence of a maximizer. In the plane we prove a continuity result for the eigenvalues imposing a bound on the number of connected components of the sequence $\Gamma_{S,n}$, obtaining in this way a version of the famous result of V. \v{S}ver\'ak (\cite{S93}) for the Steklov-Dirichlet eigenvalues. Using this result we prove the existence of a maximizer under the same topological constraint and the measure constraint.  
\end{abstract}
\tableofcontents

\section{Introduction and Main Results}
Let $\Omega\subset \mathbb{R}^d$ be a bounded, open, connected set with   Lipschitz boundary. Let  $\Gamma_S\subset \partial \Omega$ be a relative open submanifold with Lipschitz boundary and we define also $\Gamma_D=\partial \Omega\setminus \Gamma_S$. We consider the following mixed Steklov-Dirichlet eigenvalue problem:
\begin{equation}\label{eSD}
\begin{cases}
     \Delta u=0\quad  &\Omega   \\
      \partial_{\nu} u=\lambda(\Omega,\Gamma_S)u \quad &\Gamma_S \\
      u=0 \quad &\Gamma_D,
\end{cases}
\end{equation}
where $\nu$ stands for the outer unit normal. It is known (see \cite{A06}) that the Steklov-Dirichlet eigenvalue problem \eqref{eSD} has a discrete spectrum $\{\lambda_k(\Omega,\Gamma_S) \}^{\infty}_{k=1}$ 
\begin{equation*}
0< \lambda_1(\Omega,\Gamma_S)\leq \lambda_2(\Omega,\Gamma_S)\leq \lambda_3(\Omega,\Gamma_S)\leq \cdots \rightarrow +\infty,
\end{equation*}
and the eigenvalues admit the following variational characterization:
\begin{equation}\label{eVC}
\lambda_k(\Omega,\Gamma_S)= \inf_{V_k\subset H_0^1(\Omega,\Gamma_S)} \sup_{v\in V_k} \frac{\int_{\Omega}|\nabla v|^2 dx}{\int_{\Gamma_S}v^2d\mathcal{H}^{d-1}},
\end{equation}
where the infimum is taken over all $k-$dimensional subspaces $V_k$ of the space $H_0^1(\Omega,\Gamma_S)= \{ v\in H^1(\Omega) | \; v\equiv 0\; \text{on}  \; \Gamma_D \}$, where the equality on the boundary is intended in the sense of trace. We denote the corresponding eigenfunctions by $\{u_k \}^{\infty}_{k=1}$ and we know that the following relation holds
\begin{equation*}
\lambda_k(\Omega,\Gamma_S)= \frac{\int_{\Omega}|\nabla u_k|^2 dx}{\int_{\Gamma_S}u_k^2d\mathcal{H}^{d-1}}.
\end{equation*}
The main purpose of this work is to study the dependence of the eigenvalues $\lambda_k(\Omega,\Gamma_S)$ with respect to $\Gamma_S\subset \partial \Omega$.

This kind of mixed eigenvalues has been deeply studied. For instance in \cite{HL20} bounds for the Riesz mean has been obtained, in \cite{BKPS10} the authors obtained inequalities between Steklov-Dirichlet eigenvalues and Steklov-Neumann eigenvalues, in \cite{LPPS17} the authors proved a two terms asymptotic formula and in \cite{GPPS21, F21, VS20} optimization of the first Steklov-Dirichlet eigenvalue on doubly connected domains has been studied.

The Steklov-Dirichlet eigenvalues and eigenfunctions are used to model some important physical process (see \cite{LPPS17, BKPS10, B01}) e. g. they describe the stationary heat distribution in $\Omega$ when the flux through $\Gamma_S$ is proportional to the temperature itself and the part $\Gamma_D$ is kept under zero temperature. The boundary value problem \eqref{eSD} has also interesting probabilistic interpratation (see \cite{BKPS10,BK04}).

We now describe the structure of the paper and the main results. In Section \ref{sre} we recall some useful results about the regularity of solutions of mixed boundary value problems, in particular we prove that the Steklov-Dirichlet eigenfunctions are Besov functions.

In Section \ref{ssr} we prove a stability result for the Steklov-Dirichlet eigenvalues. More precisely we prove the following theorem (see Theorem \ref{tSL} in order to have more information about the constants $C_1$ and $C_2$) 
\begin{theorem}\label{tSL2}
Let $\Omega$ be a uniform $\mathcal{C}^{1,1}$ open set of $\mathbb{R}^d$, let $\Gamma_S\subset \partial \Omega$ and $\Gamma'_S\subset \partial \Omega$ two $\mathcal{C}^{1,1}$ relative open submanifolds such that $\mathcal{H}^{d-1}(\Gamma_S \cap \Gamma'_S)>0$ . We define the two sets $\Gamma_D=\partial \Omega\setminus \Gamma_S$ and $\Gamma'_D=\partial \Omega\setminus \Gamma'_S$ then 
\begin{itemize}
\item For $d\geq 3$ there exists a constant $C_1$ such that:
\begin{equation*}
|\lambda_k(\Omega,\Gamma_S)-\lambda_k(\Omega,\Gamma'_S)|\leq C_1\big (\mathcal{H}^{d-1}(\Gamma_S \triangle \Gamma'_S)^{\frac{1}{2d}}+d(\Gamma_D,\Gamma'_D)^{\frac{1}{2}}\big ).
\end{equation*} 
\item For $d=2$ there exists a constant $C_2$ such that:
\begin{equation*}
|\lambda_k(\Omega,\Gamma_S)-\lambda_k(\Omega,\Gamma'_S)|\leq C_2\big (\mathcal{H}^1(\Gamma_S \triangle \Gamma'_S)^{\frac{1}{2}}+d(\Gamma_D,\Gamma'_D)^{\frac{1}{2}}\big ).
\end{equation*} 
\end{itemize}
Where $d(\Gamma_D,\Gamma'_D)$ is the Hausdorff distance between the two sets.
\end{theorem}
In the first part of Section \ref{sep} we study maximization and minimization problems for the Steklov-Dirichlet eigenvalues, when we put a measure constraint on $\Gamma_S$. Similar problems where already studied in \cite{CU03} for Dirichlet-Neumann eigenvalues.

In particular we prove the following theorems
\begin{theorem}\label{tEM}
Let $\Omega\subset \mathbb{R}^d$ be a Lipschitz domain and let $0<m<1$ be a constant, then, for all $k$, the following variational problem 
\begin{equation*}
\inf \{ \lambda_k(\Omega,\Gamma_S) \;| \; \Gamma_S\subset \partial \Omega,\;\;  \mathcal{H}^{d-1}(\Gamma_S)=m\mathcal{H}^{d-1}(\partial \Omega) \},
\end{equation*}
has a solution.
\end{theorem} 

\begin{theorem}\label{tNEU}
Let $\Omega\subset \mathbb{R}^d$ be a Lipschitz domain and let $0<m<1$ be a constant, then the following equality holds 
\begin{equation*}
\sup \{ \lambda_k(\Omega,\Gamma_S) \;| \; \Gamma_S\subset \partial \Omega,\;\; \mathcal{H}^{d-1}(\Gamma_S)=m\mathcal{H}^{d-1}(\partial \Omega) \}=+\infty.
\end{equation*}
\end{theorem}

Then we focus on the planar case. We show an explicit example of a sequence of domains $\Gamma_{S,n}$ for which $\lim_{n\to \infty}\lambda_1(\Omega,\Gamma_{S,n})=\infty$. We see that this phenomenon is linked to the fact that the sequence $\Gamma_{S,n}$ has an unbounded number of connected component. 

Motivated by this example, in the second part of Section \ref{sep}, we study the continuity properties of the Steklov-Dirichlet eigenvalues. More precisely we prove the following theorem, that can be seen as the analogous of the famous result by V. \v{S}ver\'ak  concerning the Dirichlet eigenvalues (see \cite{HP18, S93}) for the Steklov-Dirichlet eigenvalues 
\begin{theorem}\label{tSS}
Let $\Omega\subset \mathbb{R}^2$ be a $\mathcal{C}^{1,1}$ open domain, let $\Gamma_{D,n}\subset \partial \Omega$ be a sequence of compact subdomains converging for the Hausdorff metric to a compact set $\Gamma_D\subset \partial \Omega$. We define the two sets $\Gamma_{S,n}=\partial \Omega\setminus \Gamma_{D,n}$ and $\Gamma_S=\partial \Omega\setminus \Gamma_D$, assume that the number of connected components of $\Gamma_{D,n}$ is uniformly bounded, then for all $k$
\begin{equation*}
\lambda_k(\Omega,\Gamma_{S,n})\rightarrow \lambda_k(\Omega,\Gamma_S).
\end{equation*}
\end{theorem}
Using this continuity property we prove the existence of a maximizer in the class of sets with given measure and bounded number of connected components.

\section{Regularity of eigenfunctions}\label{sre}
Let $\Omega\subset \mathbb{R}^d $ be a bounded, open, connected set. Let  $\Gamma_S\subset \partial \Omega$ be a relative open submanifolds and we define also $\Gamma_D=\partial \Omega\setminus \Gamma_S$. In order to investigate the regularity properties of the Steklov-Dirichlet eigenfunctions $u_k$ we are lead to study the following mixed boundary value problem:
\begin{equation}\label{eMBP}
\begin{cases}
     \Delta u=0\quad  &\Omega   \\
      \partial_{\nu} u=g \quad &\Gamma_S \\
      u=0 \quad &\Gamma_D,
\end{cases}
\end{equation}
in particular we want to study the regularity of the solution depending on the regularity of the function $g$.

The big issue for this type of boundary value problem comes from the singularities that can appear when the boundary conditions change. The following example shows that the solution of \eqref{eMBP} is not smooth in general, no matter how the data are regular. 
\begin{example}\label{example}
We define the following set $\mathbb{R}_+^2=\{(x,y)\in \mathbb{R}^2\,|\,y>0 \}$  and we consider the following problem
\begin{equation*}
\begin{cases}
     \Delta u=0\quad  &\mathbb{R}_+^2   \\
      \partial_{\nu} u=0 \quad &\{(x,y)\in \mathbb{R}^2\,|\,x<0,\,y=0 \} \\
      u=0 \quad &\{(x,y)\in \mathbb{R}^2\,|\,x\geq 0,\,y=0 \}.
\end{cases}
\end{equation*}
In polar coordinates a solution for the above problem is given by the following function: 
\begin{equation*}
u(r,\theta)=r^{\frac{1}{2}}\sin(\frac{\theta}{2}) \quad r\geq 0,\, 0\leq \theta \leq 2\pi.
\end{equation*}
We have that $u\in \mathcal{C}^{\frac{1}{2}}(\overline{\mathbb{R}_+^2})$.
\end{example}
In general we cannot expect that the solutions of \eqref{eMBP} are more regular than $\mathcal{C}^{\frac{1}{2}}(\overline{\Omega})$. There is a huge literature concerning the problem of regularity of solutions of mixed boundary value problems, without claiming to be exhaustive  we can refer to the following monographs \cite{D88, G85}, to the articles \cite{AK82,S97,KM07} and the references therein. We will investigate the case when both $\Omega$ and $\Gamma_S$ are smooth, and we will measure the regularity of the solutions in the class of Besov spaces, deeply using the results from \cite{S97}.   

Before presenting the results in our setting we want to comment the fact that, counterintuitively, the case when $\Omega$ is smooth is the case where the least regularity is expected. Indeed when $\Gamma_S$ and $\Gamma_D$ do not meet tangentially then we have a better situation and we have more regularity for the solutions (see\cite{D88}), but the gain of regularity vanishes as the angle between $\Gamma_S$ and $\Gamma_D$ approach $\pi$. 

We introduce the Besov space $B_{2\,\infty}^{\frac{3}{2}}(\Omega)$ via its characterization as an interpolation space between two Sobolev spaces, let $(\cdot,\cdot)_{\theta,p}$ be the real interpolation functor (see \cite{BL76}), then we have that:
\begin{equation*}
B_{2\,\infty}^{\frac{3}{2}}(\Omega)=\big (H^1(\Omega),H^2(\Omega) \big )_{\frac{1}{2},\infty}.
\end{equation*}
Under some smoothness assumption on $\Omega$ and $\Gamma_S$ we can conclude that if $u$ is a solution of \eqref{eMBP} then $u\in B_{2\,\infty}^{\frac{3}{2}}(\Omega)$ (see \cite{S97}), and we will use this results in order to show that $u_k\in B_{2\,\infty}^{\frac{3}{2}}(\Omega)$.

We define what are the precise regularity property of $\Omega$ and $\Gamma_S$ (see \cite{A75,S70,S97})
\begin{definition}\label{dOm}
We say that $\Omega$ is a uniform $\mathcal{C}^{1,1}$ open set of $\mathbb{R}^d$ and $\Gamma_S\subset \partial \Omega$ is a $\mathcal{C}^{1,1}$ relative open submanifold if there exists an $\epsilon>0$, an integer $l$, an $M>0$ and a possible finite sequence $U_1,...,U_k,...$ of open sets of $\mathbb{R}^d$ so that 
\begin{enumerate}
\item if $x\in \partial \Omega$ then $\overline{B_{\epsilon}(x)}\subset U_k$ for some $k$;
\item no point of $\mathbb{R}^d$ is contained in more than $l$ of the $U_k$'s;
\item there exist $\mathcal{C}^{1,1}$ diffeomorphisms
\begin{equation*}
\begin{cases}
      \Phi_k:U_k\rightarrow V_k=B_1(x_k)\subset \mathbb{R}^d,\quad \Psi_k=\Phi_k^{-1}, \\
      ||\Phi_k||_{\mathcal{C}^{1,1}(U_k)}\leq M \quad  ||\Psi_k||_{\mathcal{C}^{1,1}(U_k)}\leq M
\end{cases}
\end{equation*}
\end{enumerate}
with $x_k\in \partial \mathbb{R}_+^d$, and 
\begin{equation*}
\begin{cases}
      \Phi_k(U_k\cap \Omega)=V_k\cap\mathbb{R}_+^d , \\
      \Phi_k(U_k\cap \partial \Omega)=V_k\cap\partial \mathbb{R}_+^d , \\
      \Phi_k(U_k\cap \Gamma_D)=V_k\cap \mathbb{R}_+^{d-1}\times \{0\}, \\
      \Phi_k(U_k\cap \partial \Gamma_D)=V_k\cap \mathbb{R}^{d-2}\times \{(0,0)\}.
\end{cases}
\end{equation*}
\end{definition}
We can now state the result about the regularity of eigenfunctions
\begin{theorem}\label{tRE}
Let $\Omega$ be a uniform $\mathcal{C}^{1,1}$ open set of $\mathbb{R}^d$, let $\Gamma_S\subset \partial \Omega$ be a $\mathcal{C}^{1,1}$ relative open submanifold (as in Definition \ref{dOm}) and let $u_k$ be an eigenfunction corresponding to the eigenvalue $\lambda_k(\Omega,\Gamma_S)$ then $u_k\in B_{2\,\infty}^{\frac{3}{2}}(\Omega)$. Moreover if $n=2$ then $u_k\in \mathcal{C}^{\frac{1}{2}}(\overline{\Omega})$.
\end{theorem}
\begin{proof}
We know that $u_k\in H^1(\Omega)$ so in particular $u_k\in H^{\frac{1}{2}}(\Gamma_S)$ and from Theorem 1 in \cite{S97} we conclude that  $u_k\in B_{2\,\infty}^{\frac{3}{2}}(\Omega)$. Moreover if $d=2$ the Sobolev-Besov embedding Theorem gives us that $B_{2\,\infty}^{\frac{3}{2}}(\Omega)\subset \mathcal{C}^{\frac{1}{2}}(\overline{\Omega})$ and in particular $u_k\in \mathcal{C}^{\frac{1}{2}}(\overline{\Omega})$
\end{proof}

\begin{remark}
We notice that in dimension $d=2$ the Steklov-Dirichlet eigenfunctions $u_k$ reach the maximal Holder regularity allowed by Example \ref{example}. 
\end{remark}

\section{Stability of the Steklov-Dirichlet eigenvalues}\label{ssr}
In this section we investigate the stability properties of the eigenvalues $\lambda_k(\Omega,\Gamma_S)$. More precisely we state Theorem \ref{tSL2} in a more precise way, explaining the precise dependence of the constants $C_1$ and $C_2$ from $\Omega$, $\Gamma_S$ and $\Gamma'_S$, and we prove the stability result.

In order to do so we need to introduce the following Lemma about the existence of a family of test functions. 
\begin{lemma}\label{lTF}
Let $u_k$ be an eigenfunction associated to the eigenvalue $\lambda_k(\Omega,\Gamma_S)$ and let $M$ be a constant independent from $\Gamma_S$ and $\Gamma'_S$,  then there exist a family of functions $g_i\in \mathcal{C}^{\infty}(\Gamma'_S \setminus \Gamma_S)$ with $i=1,...,k$, such that $||g_i||_{\mathcal{C}^0(\Omega)}\leq M$ and the solutions of the following mixed boundary value
\begin{equation}\label{eTS}
\begin{cases}
     \Delta \widetilde{v}_i=0\quad  &\Omega   \\
      \partial_{\nu} \widetilde{v}_i=\lambda_k(\Omega,\Gamma_S)u_k \quad &\Gamma'_S\cap \Gamma_S \\
      \partial_{\nu} \widetilde{v}_i=g_i \quad &\Gamma'_S\setminus\Gamma_S \\
      \widetilde{v}_i=0 \quad &\partial \Omega\setminus \Gamma'_S .
\end{cases}
\end{equation}
have the following properties: for all $i\neq j$ we have that
\begin{equation}\label{eCTF}
\int_{\Omega} \nabla \widetilde{v}_i \cdot \nabla \widetilde{v}_j dx=\int_{\partial \Omega} \widetilde{v}_i \widetilde{v}_j d\mathcal{H}^{d-1}=0.
\end{equation}
The same result holds if we exchange the role of $\Gamma_S$ and $\Gamma'_S$
\end{lemma}
\begin{proof}
We define the functions $g_i$ via an induction procedure. In the proof $M$ will be a constant independent from $\Gamma_S$ and $\Gamma'_S$ that can change from line to line. We choose the first function $g_1\in \mathcal{C}^{\infty}(\Gamma'_S \setminus \Gamma_S)$ such that $||g_1||_{C^0(\Gamma'_S \setminus \Gamma_S)}\leq M$ and such that the function
\begin{equation*}
\Phi_1(x)=\begin{cases}
     \lambda_k(\Omega,\Gamma_S)u_k(x) \quad  &x\in\Gamma_S \cap \Gamma'_S   \\
      g_1(x) \quad &x\in\Gamma'_S \setminus \Gamma_S
\end{cases}
\end{equation*}
is in $H^{\frac{1}{2}}(\Gamma'_S)$, so we introduce the first test function $\widetilde{v}_1$
\begin{equation*}
\begin{cases}
     \Delta \widetilde{v}_1=0\quad  &\Omega   \\
      \partial_{\nu} \widetilde{v}_1=\lambda_k(\Omega,\Gamma_S)u_k \quad &\Gamma'_S\cap \Gamma_S \\
      \partial_{\nu} \widetilde{v}_1=g_1 \quad &\Gamma'_S\setminus\Gamma_S \\
      \widetilde{v}_1=0 \quad &\partial \Omega\setminus \Gamma'_S .
\end{cases}
\end{equation*}
We now show how to construct the function $g_2$ with the desired properties, and then we generalize this procedure, constructing the function $g_{i+1}$ knowing the function $g_1,g_2,...,g_i$. We introduce two functions $\varphi_1$ and $\varphi_2$ in $\mathcal{C}_c^{\infty}(\Gamma'_S \setminus \Gamma_S)$ such that $||\varphi_1||_{C^0(\Gamma'_S \setminus \Gamma_S)}\leq M$ and $||\varphi_2||_{C^0(\Gamma'_S \setminus \Gamma_S)}\leq M$, two real parameters $t_1$ and $t_2$, we will choose more precisely the functions and the parameters later in the proof. We define also $v_{\varphi_i}$ with $i=1,2$ to be the solution of the following mixed boundary value problem
\begin{equation*}
\begin{cases}
     \Delta v_{\varphi_i}=0\quad  &\Omega   \\
      \partial_{\nu} v_{\varphi_i}=\varphi_i \quad &\Gamma'_S \\
      v_{\varphi_i}=0 \quad &\partial \Omega\setminus \Gamma'_S .
\end{cases}
\end{equation*}  
We denote by $\widetilde{v}_2$ the solution of the following problem:
\begin{equation}\label{eWTF}
\begin{cases}
     \Delta \widetilde{v}_2=0\quad  &\Omega   \\
      \partial_{\nu} \widetilde{v}_2=\lambda_k(\Omega,\Gamma_S)u_k \quad &\Gamma'_S\cap \Gamma_S \\
      \partial_{\nu} \widetilde{v}_2=g_1+t_1\varphi_1+t_2\varphi_2 \quad &\Gamma'_S\setminus\Gamma_S \\
      \widetilde{v}_2=0 \quad &\partial \Omega\setminus \Gamma'_S .
\end{cases}
\end{equation}
Now we want to choose the functions $\varphi_1$ and $\varphi_2$ and the parameters $t_1$ and $t_2$ in such a way that the conditions \eqref{eCTF} are satisfied. More precisely, from the linearity of the equation \eqref{eWTF} and from the linearity of the Dirichlet-to-Neumann map, we obtain the following equalities:
\begin{align*}
\int_{\Omega} \nabla \widetilde{v}_1 \cdot \nabla \widetilde{v}_2 dx&=\int_{\Gamma'_S\cap \Gamma_S} \lambda_k(\Omega,\Gamma_S)^2u_k^2 d\mathcal{H}^{d-1}+\int_{\Gamma'_S\setminus \Gamma_S}g_1^2 d\mathcal{H}^{d-1}+\\
&+t_1\int_{\Gamma'_S\setminus \Gamma_S}g_1\varphi_1 d\mathcal{H}^{d-1}+t_2\int_{\Gamma'_S\setminus \Gamma_S}g_1\varphi_2 d\mathcal{H}^{d-1}, \\
\int_{\partial \Omega}\widetilde{v}_1\widetilde{v}_2dx&=\int_{\Gamma'_S} \widetilde{v}_1^2d\mathcal{H}^{d-1}+t_1\int_{\Gamma'_S} \widetilde{v}_1v_{\varphi_1}d\mathcal{H}^{d-1}+t_2\int_{\Gamma'_S} \widetilde{v}_1v_{\varphi_2}d\mathcal{H}^{d-1}.
\end{align*} 
Now we introduce the following matrix and the following vector 
\begin{equation*}
A_{2\times 2}=\begin{pmatrix}
\int_{\Gamma'_S\setminus \Gamma_S}g_1\varphi_1 d\mathcal{H}^{d-1} & \int_{\Gamma'_S\setminus \Gamma_S}g_1\varphi_2 d\mathcal{H}^{d-1} \\
\int_{\Gamma'_S} \widetilde{v}_1v_{\varphi_1}d\mathcal{H}^{d-1} & \int_{\Gamma'_S} \widetilde{v}_1v_{\varphi_2}d\mathcal{H}^{d-1}, 
\end{pmatrix}
\end{equation*}
\begin{equation*}
b_{2}=\begin{pmatrix}
- \int_{\Gamma'_S\cap \Gamma_S} \lambda_k(\Omega,\Gamma_S)^2u_k^2 d\mathcal{H}^{d-1}-\int_{\Gamma'_S\setminus \Gamma_S}g_1^2 d\mathcal{H}^{d-1} \\
-\int_{\Gamma'_S} \widetilde{v}_1^2d\mathcal{H}^{d-1}.
\end{pmatrix}
\end{equation*}
We choose $\varphi_1$ and $\varphi_2$ in such a way that $A_{2\times 2}^{-1}$ exists, $||A_{2\times 2}^{-1}||\leq M$ and $||b_2||\leq M$ (we recall that the constant $M$ can change). Now we choose the parameters $t=(t_1,t_2)$ to be the solutions of the following linear system $A_{2\times 2}t=b_2$. Now we choose $g_2$ to be the following function:
\begin{equation*}
g_2=g_1+t_1\varphi_1+t_2\varphi_2.
\end{equation*}
With the choice we made for $\varphi_{1}$, $\varphi_{2}$ and $t=(t_1,t_2)$ it is clear that $||g_2||_{\mathcal{C}^0(\Omega)}\leq M$ and the function $\widetilde{v}_2$ satisfies the conditions \eqref{eCTF}.

Now suppose we know the functions $g_1,...,g_i$, we can construct the function $g_{i+1}$ using the same algorithm. This time we need to introduce the functions $\{\varphi_j\}_{j=1}^{2i}$ such that $\varphi_j\in \mathcal{C}_c^{\infty}(\Gamma'_S \setminus \Gamma_S)$ and $||\varphi_j||_{C^0(\Gamma'_S \setminus \Gamma_S)}\leq M$ for all $j=1,...,2i$, and the parameters $\{t_j\}_{j=1}^{2i}$. We introduce the function 
\begin{equation*}
\begin{cases}
     \Delta \widetilde{v}_{i+1}=0\quad  &\Omega   \\
      \partial_{\nu} \widetilde{v}_{i+1}=\lambda_k(\Omega,\Gamma_S)u_k \quad &\Gamma'_S\cap \Gamma_S \\
      \partial_{\nu} \widetilde{v}_{i+1}=g_1+\sum_{j=1}^{2i}t_j\varphi_j. \quad &\Gamma'_S\setminus\Gamma_S \\
      \widetilde{v}_{i+1}=0 \quad &\partial \Omega\setminus \Gamma'_S .
\end{cases}
\end{equation*}
Now, by imposing $2i$ orthogonality conditions with the function $\{\widetilde{v}_{j}\}_{j=1}^i$ we will obtain a similar system of linear equations, in particular we will define in a similar way the matrix $A_{2i\times 2i}$ and the vector $b_{2i}$. We choose the functions $\{\varphi_j\}_{j=1}^{2i}$ in such a way that $A_{2i\times 2i}^{-1}$ exists, $||A_{2i\times 2i}^{-1}||\leq M$ and $||b_{2i}||\leq M$ (we recall that the constant $M$ can change). Now we choose the vector $t=(t_1,t_2,...,t_{2i})$ to be the solution of the following system $A_{2i\times 2i}t=b_{2i}$ and we define $g_{i+1}$ in the following way
\begin{equation*}
g_{i+1}=g_1+\sum_{j=1}^{2i}t_j\varphi_j.
\end{equation*}
With the choice we made for $\{\varphi_j\}_{j=1}^{2i}$ and $t=(t_1,t_2,...,t_{2i})$ it is clear that $||g_{i+1}||_{\mathcal{C}^0(\Omega)}\leq M$ and the function $\widetilde{v}_{i+1}$ satisfies the conditions \eqref{eCTF}. The same construction is possible when we exchange the two sets $\Gamma_S$ and $\Gamma'_S$.
\end{proof}

We are now ready to state in a more precise way Theorem \ref{tSL2} and to prove it. Let $K_1\subset \mathbb{R}^d$ and $K_2\subset \mathbb{R}^d$ be two compact sets, we denote by $d(K_1,K_2)$ the Hausdorff distance between the two sets. 
 
\begin{theorem}\label{tSL}
Let $\Omega$ be a uniform $\mathcal{C}^{1,1}$ open set of $\mathbb{R}^d$, let $\Gamma_S\subset \partial \Omega$ and $\Gamma'_S\subset \partial \Omega$ two $\mathcal{C}^{1,1}$ relative open submanifolds (see Definition \ref{dSET}) such that $\mathcal{H}^{d-1}(\Gamma_S \cap \Gamma'_S)>0$ . We define the two sets $\Gamma_D=\partial \Omega\setminus \Gamma_S$ and $\Gamma'_D=\partial \Omega\setminus \Gamma'_S$ and let $v_k$ be a linear combination of the functions $\{\widetilde{v}_i\}_{i=1}^k$ defined in Lemma \ref{lTF} then 
\begin{itemize}
\item For $d\geq 3$ there exists a constant $C_1$, that depends on $\Omega$, on $||v_k||_{L^2(\partial \Omega)}$, on $||v_k||_{H^1(\Omega)}$ and on $\max \{\lambda_k(\Omega,\Gamma_S),\lambda_k(\Omega,\Gamma'_S) \}$,  such that:
\begin{equation}
|\lambda_k(\Omega,\Gamma_S)-\lambda_k(\Omega,\Gamma'_S)|\leq C_1\big (\mathcal{H}^{d-1}(\Gamma_S \triangle \Gamma'_S)^{\frac{1}{2d}}+d(\Gamma_D,\Gamma'_D)^{\frac{1}{2}}\big ).
\end{equation} 
\item For $d=2$ there exists a constant $C_1$, that depends on $\Omega$, on $||v_k||_{L^2(\partial \Omega)}$, on $||v_k||_{H^1(\Omega)}$ and on $\max \{\lambda_k(\Omega,\Gamma_S),\lambda_k(\Omega,\Gamma'_S) \}$,  such that:
\begin{equation}
|\lambda_k(\Omega,\Gamma_S)-\lambda_k(\Omega,\Gamma'_S)|\leq C_2\big (\mathcal{H}^1(\Gamma_S \triangle \Gamma'_S)^{\frac{1}{2}}+d(\Gamma_D,\Gamma'_D)^{\frac{1}{2}}\big ).
\end{equation} 
\end{itemize}
\end{theorem}
We will prove the upper bound for the quantity $\lambda_k(\Omega,\Gamma'_S)-\lambda_k(\Omega,\Gamma_S)$, from the proof of this upper bound we will easily obtain the desired estimates also for the quantity $\lambda_k(\Omega,\Gamma_S)-\lambda_k(\Omega,\Gamma'_S)$ by simply exchanging the role of $\Gamma_S$ and $\Gamma'_S$.

\begin{proof} 
In the proof we will denote by $C$ a constant, which can change from line to line, that depends on $\Omega$, on $||v_k||_{L^2(\partial \Omega)}$, on $||v_k||_{H^1(\Omega)}$  and on $\max \{\lambda_k(\Omega,\Gamma_S),\lambda_k(\Omega,\Gamma'_S) \}$. We start by estimating the following quantity:
\begin{equation}
\lambda_k(\Omega,\Gamma'_S)-\lambda_k(\Omega,\Gamma_S).
\end{equation}
We use the variational characterization \eqref{eVC} and the test functions constructed in Lemma \ref{lTF} and we obtain that 
\begin{equation*}
\lambda_k(\Omega,\Gamma'_S)\leq \max_{\alpha \in \mathbb{R}^k}  \frac{\int_{\Omega}|\nabla \big (\sum_{i=1}^k \alpha_i\widetilde{v}_i\big )|^2 dx}{\int_{\Gamma'_S}\big (\sum_{i=1}^k \alpha_i\widetilde{v}_i\big )^2d\mathcal{H}^{d-1}}=\max_{\alpha \in \mathbb{R}^k} \frac{\sum_{i=1}^k \alpha_i^2 \int_{\Omega}|\nabla \widetilde{v}_i|^2 dx}{\sum_{i=1}^k \alpha_i^2 \int_{\Gamma'_S}\widetilde{v}_i^2d\mathcal{H}^{d-1}},
\end{equation*}
where the last equality is true because of the conditions \eqref{eCTF}. Let $\widetilde{\alpha}\in \mathbb{R}^k$ be the solution of the maximization problem, from the last equality it is clear that we can assume $\widetilde{\alpha}_i\geq 0$ for all $i=1,...,k$, in particular $\sum_{i=1}^k \widetilde{\alpha}_i>0$, because otherwise $\widetilde{\alpha}_i=0$ for all $i=1,...,k$ and this is not possible. The Rayleigh quotient is invariant under multiplication by a scalar, so we can also assume that $\sum_{i=1}^k \widetilde{\alpha}_i=1$ and we define the following function:  
\begin{equation*}
v_k=\sum_{i=1}^k \widetilde{\alpha}_i\widetilde{v}_i.
\end{equation*}
From the linearity of the equations \eqref{eTS} we know that $v_k$ satisfies the following mixed boundary value problem
\begin{equation}\label{eTF}
\begin{cases}
     \Delta v_k=0\quad  &\Omega   \\
      \partial_{\nu} v_k=\lambda_k(\Omega,\Gamma_S)u_k \quad &\Gamma'_S\cap \Gamma_S \\
      \partial_{\nu}v_k=\sum_{i=1}^k \widetilde{\alpha}_ig_i \quad &\Gamma'_S\setminus\Gamma_S \\
      v_k=0 \quad &\partial \Omega\setminus \Gamma'_S .
\end{cases}
\end{equation}

From the variational characterization, using the function $v_k$ as a test function for $\lambda_k(\Omega,\Gamma'_S)$ and let $u_k$ be an eigenfunction corresponding to $\lambda_k(\Omega,\Gamma_S)$, we obtain the following inequality: 
\begin{equation}\label{eSSD1}
\lambda_k(\Omega,\Gamma'_S)-\lambda_k(\Omega,\Gamma_S)\leq \frac{\int_{\Gamma_S}u_k^2d\mathcal{H}^{d-1}\int_{\Omega}|\nabla v_k|^2 dx-\int_{\Gamma'_S}v_k^2d\mathcal{H}^{d-1}\int_{\Omega}|\nabla u_k|^2 dx}{\int_{\Gamma_S}u_k^2d\mathcal{H}^{d-1}\int_{\Gamma'_S}v_k^2d\mathcal{H}^{d-1}}.
\end{equation}
We start by giving an upper bound for the quantity $\int_{\Gamma_S}u_k^2$. From Cauchy-Schwarz inequality we know that
\begin{equation*}
\int_{\Gamma_S}u_k^2d\mathcal{H}^{d-1}-\int_{\Gamma'_S}v_k^2d\mathcal{H}^{d-1} \leq ||u_k+v_k||_{L^2(\partial \Omega)}||u_k-v_k||_{L^2(\partial \Omega)},
\end{equation*} 
and we also know that 
\begin{equation*}
||v_k||_{L^2(\partial \Omega)} \leq ||u_k-v_k||_{L^2(\partial \Omega)}+||u_k||_{L^2(\partial \Omega)},
\end{equation*}
so we conclude that 
\begin{equation}\label{eDifUUK}
\int_{\Gamma_S}u_k^2d\mathcal{H}^{d-1}-\int_{\Gamma'_S}v_k^2d\mathcal{H}^{d-1} \leq \big ( ||u_k-v_k||_{L^2(\partial \Omega)}+2||u_k||_{L^2(\partial \Omega)} \big )||u_k-v_k||_{L^2(\partial \Omega)}.
\end{equation}
From the regularity assumption on the domain $\Omega$ we know that the trace operator is compact and we denote its norm by $C_t$. Using this fact, combined with \eqref{eDifUUK} we obtain the following upper bound for $\int_{\Gamma_S}u_k^2$:
\begin{equation*}
\int_{\Gamma_S}u_k^2d\mathcal{H}^{d-1} \leq \big ( C_t||u_k-v_k||_{H^1(\Omega)}+2||u_k||_{L^2(\partial \Omega)} \big )C_t||u_k-v_k||_{H^1(\Omega)}+\int_{\Gamma'_S}v_k^2d\mathcal{H}^{d-1},
\end{equation*}
using this estimate in \eqref{eSSD1} we finally obtain 
\begin{align}\label{eSSD2}
\lambda_k(\Omega,\Gamma'_S)- \lambda_k(\Omega,\Gamma_S)&\leq 
 \frac{1}{\int_{\Gamma_S}u_k^2d\mathcal{H}^{d-1}\int_{\Gamma'_S}v_k^2d\mathcal{H}^{d-1}}\Big [\int_{\Gamma'_S}v_k^2d\mathcal{H}^{d-1}\big (\int_{\Omega}|\nabla v_k|^2 dx-\int_{\Omega}|\nabla u_k|^2 dx\big ) +\\ \notag
 &+\big ( C_t^2||u_k-v_k||_{H^1(\Omega)}^2+2||u_k||_{L^2(\partial \Omega)}C_t||u_k-v_k||_{H^1(\Omega)} \big )\int_{\Omega}|\nabla v_k|^2 dx \Big ].
\end{align}
We want to bound the right hand side of \eqref{eSSD2}. We start by noticing that the function $v_k$ is in $H^1(\Omega)$ and also that there exists a constant $C>0$ such that 
\begin{equation}\label{eLBU}
C\leq\int_{\Gamma'_S} v_k^2d\mathcal{H}^{d-1},
\end{equation}
indeed if we have that $\int_{\Gamma'_S} v_k^2d\mathcal{H}^{d-1}=0$, we will have that the trace of the function $v_k$ is a function constantly equal to zero, and this is a contradiction with the equation \eqref{eTF} that $v_k$ must satisfy.

We now estimate the quantity $\int_{\Omega}|\nabla v_k|^2 dx-\int_{\Omega}|\nabla u_k|^2 dx$ in \eqref{eSSD2}, we obtain that 
\begin{align}\label{eUBIfGrad}
&\int_{\Omega}|\nabla v_k|^2 dx-\int_{\Omega}|\nabla u_k|^2 dx\leq \big | ||v_k||^2_{H^1(\Omega)}-||u_k||^2_{H^1(\Omega)}\big |+\big | ||v_k||^2_{L^2(\Omega)}-||u_k||^2_{L^2(\Omega)} \big |\\\notag
&\leq \big | ||v_k||_{H^1(\Omega)}-||u_k||_{H^1(\Omega)}\big |\big | ||v_k||_{H^1(\Omega)}+||u_k||_{H^1(\Omega)}\big |+\big | ||v_k||_{L^2(\Omega)}-||u_k||_{L^2(\Omega)} \big |\big | ||v_k||_{L^2(\Omega)}+||u_k||_{L^2(\Omega)} \big |\\\notag
&\leq C ||u_k-v_k||_{H^1(\Omega)}.
\end{align}
 The last remaining term to estimate is the quantity $||u_k-v_k||_{H^1(\Omega)}$, in order to estimate this term we use Theorem 4 in \cite{S97} and \eqref{eTF} to conclude that there exists a constant $C$ that depends only on $\Omega$ such that the following inequality holds
\begin{align*}
||u_k-&v_k||_{H^1(\Omega)}\leq C \Big (\sup_{||w||_{H^{\frac{1}{2}}(\partial \Omega)}=1}\big | \lambda_k(\Omega,\Gamma_S)\int_{\Gamma_S}(\partial_{\nu}u_k)wd\mathcal{H}^{d-1}-\int_{\Gamma'_S}(\partial_{\nu} v_k) wd\mathcal{H}^{d-1} \big | +d(\Gamma_D,\Gamma'_D)^{\frac{1}{2}} \Big)\\
&\leq C \Big ( \sup_{||w||_{H^{\frac{1}{2}}(\partial \Omega)}=1}\int_{\Gamma_S\setminus \Gamma'_S}|u_kw|d\mathcal{H}^{d-1} +\int_{\Gamma'_S\setminus \Gamma_S}|w \big ( \sum_{i=1}^k \widetilde{\alpha}_ig_i\big )|d\mathcal{H}^{d-1}+d(\Gamma_D,\Gamma'_D)^{\frac{1}{2}} \Big)\\
&\leq C \big( ||u_k||_{L^2(\Gamma_S\setminus \Gamma'_S)}+||\sum_{i=1}^k\widetilde{\alpha}_ig_i||_{L^2(\Gamma'_S\setminus \Gamma_S)}+d(\Gamma_D,\Gamma'_D)^{\frac{1}{2}}\big).
\end{align*}
Let us now consider separately the case $d\geq 3$ and $d=2$. For $d\geq 3$ it is enough to know that $u_k\in H^\frac{1}{2}(\partial \Omega)$, indeed by classical embedding theorem for Sobolev spaces we have that $u_k\in L^{\frac{2d}{d-1}}(\partial \Omega)$ by Lemma \ref{lTF} we also know that $\sum_{i=1}^k\widetilde{\alpha}_ig_i\in L^{\frac{2d}{d-1}}(\partial \Omega)$. Using H\"older inequality we conclude that 
\begin{equation}\label{eUBHNp}
||u_k-v_k||_{H^1(\Omega)}\leq C\big( ||u_k||_{L^{\frac{2d}{d-1}}(\partial  \Omega)}\mathcal{H}^{d-1}(\Gamma_S\setminus \Gamma'_S)^{\frac{1}{2d}}+||\sum_{i=1}^k\widetilde{\alpha}_ig_i||_{L^{\frac{2d}{d-1}}(\partial  \Omega)}\mathcal{H}^{d-1}(\Gamma'_S\setminus \Gamma_S)^{\frac{1}{2d}}+d(\Gamma_D,\Gamma'_D)^{\frac{1}{2}}\big)
\end{equation}
From Lemma \ref{lTF} we know that there exists a constant $C$, that does not depend on $\Gamma_S$ and $\Gamma'_S$, such that $||\sum_{i=1}^k\widetilde{\alpha}_ig_i||^{\frac{1}{2}}_{L^{\frac{2d}{d-1}}(\partial  \Omega)}\leq C$. From the continuity of the Sobolev embedding and the trace operator we have that there exists a constant $C$, that depends only on $\Omega$, such that $||u_k||_{L^{\frac{2d}{d-1}}(\partial  \Omega)}\leq C ||u_k||_{H^1(\Omega)}$. Now using this inequality together with the Poincar\'e -Friedrichs inequality, assuming that $||u_k||_{L^2(\Omega)}=1$, we obtain
\begin{equation*}
||u_k||_{L^{\frac{2d}{d-1}}(\partial  \Omega)}\leq C \sqrt{\lambda_k(\Omega,\Gamma_S)\big (\frac{1}{\lambda_{1,1}(\Omega)}+1\big)+\frac{1}{\lambda_{1,1}(\Omega)}},
\end{equation*}  
where $\lambda_{1,1}(\Omega)$ is the first Robin eigenvalue with parameter $1$. From the inequality above and from \eqref{eUBHNp} we conclude that there exists a constant $C$ that depends only on $\Omega$ and $\lambda_k(\Omega,\Gamma_S)$ such that
\begin{equation}\label{eUBHN}
||u_k-v_k||_{H^1(\Omega)}\leq  C\big (\mathcal{H}^{d-1}(\Gamma_S \triangle \Gamma'_S)^{\frac{1}{2d}}+d(\Gamma_D,\Gamma'_D)^{\frac{1}{2}}\big ).
\end{equation}
For $d=2$, using the regularity results given in Theorem \ref{tRE}, we know that $u_k\in\mathcal{C}^{\frac{1}{2}}(\overline{\Omega})$, by Lemma \ref{lTF} we also know that $\sum_{i=1}^k\widetilde{\alpha}_ig_i \in\mathcal{C}^{\frac{1}{2}}(\partial \Omega)$, so we obtain that

\begin{equation}\label{eUBHN2p}
||u_k-v_k||_{H^1(\Omega)}\leq C\big( ||u_k||_{\mathcal{C}^0(\overline{\Omega})}^{\frac{1}{2}}\mathcal{H}^1(\Gamma_S \setminus \Gamma'_S)^{\frac{1}{2}}+||\sum_{i=1}^k\widetilde{\alpha}_ig_i||_{\mathcal{C}^0(\partial \Omega)}^{\frac{1}{2}}\mathcal{H}^1(\Gamma'_S \setminus \Gamma_S)^{\frac{1}{2}}+d(\Gamma_D,\Gamma'_D)^{\frac{1}{2}}\big)
\end{equation}
From Lemma \ref{lTF} we know that there exists a constant $C$, that does not depend on $\Gamma_S$ and $\Gamma'_S$, such that $||\sum_{i=1}^k\widetilde{\alpha}_ig_i||_{\mathcal{C}^0(\partial \Omega)}\leq C$. From the continuity of the Sobolev-Besov embedding (see \cite{A75,BL76,S97}) we know that there exists a constant $C$, that depends only on $\Omega$, such that $||u_k||_{\mathcal{C}^{\frac{1}{2}}(\overline{\Omega})}\leq C ||u_k||_{ B_{2\,\infty}^{\frac{3}{2}}(\Omega)}$. From this inequality and from Theorem $1$ in \cite{S97} we conclude that there exists a constant $C$, that depends only on $\Omega$, such that
\begin{equation*}
||u_k||_{\mathcal{C}^{\frac{1}{2}}(\overline{\Omega})}\leq C \sqrt{||u_k||_{H^1(\Omega)}(1+||u_k||_{H^1(\Omega)})}.
\end{equation*}  
Now using a similar argument as the one we used in the case $d\geq 3$ we can bound from above the quantity $||u_k||_{H^1(\Omega)}$ and, from \eqref{eUBHN2p}, we conclude that there exists a constant $C$, that depends only on $\Omega$ and $\lambda_k(\Omega,\Gamma_S)$, such that
\begin{equation}\label{eUBHN2}
||u_k-v_k||_{H^1(\Omega)}\leq C\big (\mathcal{H}^1(\Gamma_S \triangle \Gamma'_S)^{\frac{1}{2}}+d(\Gamma_D,\Gamma'_D)^{\frac{1}{2}}\big ).
\end{equation}

For $d\geq 3$ using \eqref{eLBU}, \eqref{eUBIfGrad} and \eqref{eUBHN} in \eqref{eSSD2} we conclude that there exists a constant $C_1$, that depends on $\Omega$, on $||v_k||_{L^2(\partial \Omega)}$, on $||v_k||_{H^1(\Omega)}$ and on $\lambda_k(\Omega,\Gamma_S)$,  such that:
\begin{equation*}
\lambda_k(\Omega,\Gamma'_S)- \lambda_k(\Omega,\Gamma_S)\leq C_1\big (\mathcal{H}^{d-1}(\Gamma_S \triangle \Gamma'_S)^{\frac{1}{2d}}+d(\Gamma_D,\Gamma'_D)^{\frac{1}{2}}\big ).
\end{equation*}
For $d=2$ using using \eqref{eLBU}, \eqref{eUBIfGrad} and \eqref{eUBHN2} in \eqref{eSSD2}  we conclude that there exists a constant $C_2$, that depends on $\Omega$, on $||v_k||_{L^2(\partial \Omega)}$, on $||v_k||_{H^1(\Omega)}$ and on $\lambda_k(\Omega,\Gamma_S)$,  such that:
\begin{equation*}
\lambda_k(\Omega,\Gamma'_S)- \lambda_k(\Omega,\Gamma_S)\leq C_2\big (\mathcal{H}^1(\Gamma_S \triangle \Gamma'_S)^{\frac{1}{2}}+d(\Gamma_D,\Gamma'_D)^{\frac{1}{2}}\big ).
\end{equation*}

Now exchanging the role of $\Gamma_S$ and $\Gamma'_S$ we obtain the same kind of estimates for the quantity $\lambda_k(\Omega,\Gamma_S)- \lambda_k(\Omega,\Gamma'_S)$, this concludes the proof
\end{proof}
\section{Optimization problems and continuity properties}\label{sep}
In the first part of this section we study maximization and minimization problems for the Steklov-Dirichlet eigenvalues with a measure constraint on the set $\Gamma_S$. In the second part we study the continuity properties of the Steklov-Dirichlet eigenvalues.
\subsection{Optimization problems}
We start by proving the existence of a minimizer, in particular we prove Theorem \ref{tEM}. We now state an elementary lemma about the monotonicity property of the Steklov-Dirichlet eigenvalues. 
\begin{lemma}\label{lME}
Let $\Omega\subset \mathbb{R}^d$ be a Lipschitz domain and let $\Gamma_S\subset \partial \Omega$ and $\Gamma'_S\subset \partial \Omega$ two relative open subsets such that $\Gamma_S\subset \Gamma'_S$ then, for all $k$, the following inequality holds
\begin{equation*}
\lambda_k(\Omega,\Gamma'_S)\leq \lambda_k(\Omega,\Gamma_S)
\end{equation*}
\end{lemma}
\begin{proof}
Let $u_i$ be an eigenfunction associated to $\lambda_i(\Omega,\Gamma_S)$ with $i=1,...,k$. From the assumption $\Gamma_S\subset \Gamma'_S$, it is clear that  $V=span[u_1,u_2,...,u_k]$ is a subspace of the Hilbert space $H_0^1(\Omega,\Gamma'_S)$, in particular we can use the test subspace $V$ in the variational characterization of $\lambda_k(\Omega, \Gamma'_S)$ and this concludes the proof. 
\end{proof}
This simple lemma is crucial in order to prove the existence of a minimizer. Indeed, using this property, we can construct a minimizer by using a classical procedure based on the concept of weak $\gamma$-convergence (see \cite{BB05} for more details).  
\begin{proof}[Proof of Theorem \ref{tEM}]
We consider a minimizing sequence $\Gamma_{S,\epsilon}$, and we consider the normalized eigenfunctions $||u_{k,\epsilon}||_{L^2(\partial \Omega)}=1$. We want to obtain a uniform bound in $H^1(\Omega)$ for the eigenfunctions $u_{k,\epsilon}$. From the minimality assumption for the sequence $\Gamma_{S,\epsilon}$ we can assume that:
\begin{equation*}
\int_{\Omega}|\nabla u_{k,\epsilon}|^2 dx=\lambda_k(\Omega,\Gamma_{S,\epsilon})\leq C<\infty\quad \forall\, \epsilon>0.
\end{equation*}
Now the bound for $\int_{\Omega}u_{k,\epsilon}^2dx$ follows directly from Poincare-Friedrichs inequality 
\begin{equation*}
\int_{\Omega}u_{k,\epsilon}^2dx\leq \frac{1}{\lambda_{1,1}(\Omega)}\big [\int_{\Omega}|\nabla u_{k,\epsilon}|^2 dx+\int_{\partial \Omega} u_{k,\epsilon}^2d\mathcal{H}^{d-1}\big ],
\end{equation*}
where $\lambda_{1,1}(\Omega)$ is the first Robin eigenvalue with parameter $1$.

We conclude that there exists a function $u_k\in H^1(\Omega)$ such that, up to a subsequence, the following convergences hold:
\begin{align}\label{cCE}
u_{k,\epsilon} &\rightharpoonup u_k \quad \text{in} \quad H^1(\Omega),\\ \notag
u_{k,\epsilon} &\rightarrow u_k \quad \text{in} \quad L^2(\partial \Omega),\\ \notag
u_{k,\epsilon} &\rightarrow u_k \quad \text{a. e. in} \quad \partial \Omega.
\end{align}
Let $j\neq k$, we can use the same argument as above for $u_{j,\epsilon}$ and we obtain that there exists a function $u_j\in H^1(\Omega)$ such that the same convergences above held. Now from the orthogonality of the eigenfunctions in $L^2(\partial \Omega)$ we have that
\begin{equation}\label{eOE}
0=\int_{\partial \Omega}u_{j,\epsilon}u_{k,\epsilon}\rightarrow \int_{\partial \Omega}u_ju_k=0
\end{equation}
Now we want to prove the orthogonality of the gradient of the eigenfunctions. We can use the concept of compensated compactness (see \cite{M78}), indeed we have that div$(\nabla u_{j,\epsilon})=\Delta u_{j,\epsilon}=0$ and rot$(\nabla u_{k,\epsilon})=0$, in particular, from the convergences \eqref{cCE}, we can conclude that 
\begin{equation}\label{eOGE}
 0=\int_{\Omega}\nabla u_{k,\epsilon}\cdot \nabla u_{j,\epsilon} dx\rightarrow  \int_{\Omega}\nabla u_k\cdot \nabla u_j dx=0.
\end{equation}

From the convergence of the eigenfunctions $u_{i,\epsilon}$, with $i=1,...,k$ we can conclude that, up to a subsequence 
\begin{equation}\label{cCSE}
\sum_{i=1}^ku_{i,\epsilon}^2\rightarrow \sum_{i=1}^ku_i^2 \quad \text{a. e. in} \quad \partial \Omega.
\end{equation}
We define the following set 
\begin{equation}\label{dSET}
\Gamma=\{x\in \partial \Omega \,|\,\sum_{i=1}^ku_i(x)^2>0\}.
\end{equation}
We can use the test subspace $V=span[u_1,u_2,...,u_k] $ in the variational characterization \eqref{eVC} for the eigenvalue $\lambda_k(\Omega,\Gamma)$. Recalling the orthogonality conditions \eqref{eOE}, \eqref{eOGE}, the normalization $||u_{i,\epsilon}||_{L^2(\partial \Omega)}=1$ and the convergences \eqref{cCE} we obtain
\begin{align*}
\lambda_k(\Omega,\Gamma)&\leq \max_{\alpha \in \mathbb{R}^k \atop \sum_{i=1}^k \alpha_i^2=1}  \frac{\int_{\Omega}|\nabla \big (\sum_{i=1}^k \alpha_iu_i\big )|^2 dx}{\int_{\Gamma}\big (\sum_{i=1}^k \alpha_iu_i\big )^2d\mathcal{H}^{d-1}}\\
&\leq \sum_{i=1}^k \overline{\alpha}_i^2\int_{\Omega}|\nabla u_i|^2 dx\\
&\leq\liminf_{\epsilon\rightarrow 0} \sum_{i=1}^k \overline{\alpha}_i^2\int_{\Omega}|\nabla u_{i,\epsilon}|^2 dx\\
&\leq\liminf_{\epsilon\rightarrow 0} \sum_{i=1}^k \overline{\alpha}_i^2\lambda_i(\Omega,\Gamma_{S,\epsilon})\\
&\leq \liminf_{\epsilon\rightarrow 0} \lambda_k(\Omega,\Gamma_{S,\epsilon}).
\end{align*}
From the definition \eqref{dSET}, the convergence \eqref{cCSE} and from Fatou Lemma we have:
\begin{equation*}
\mathcal{H}^{d-1}(\Gamma)\leq \liminf_{\epsilon\rightarrow 0} \mathcal{H}^{d-1}(\{x\in \partial \Omega \,|\,\sum_{i=1}^ku_{i,\epsilon}(x)^2>0\})=m\mathcal{H}^{d-1}(\partial \Omega).
\end{equation*}
Now if $\mathcal{H}^{d-1}(\Gamma)=m\mathcal{H}^{d-1}(\partial \Omega)$ the proof is finished. If instead  $\mathcal{H}^{d-1}(\Gamma)<m\mathcal{H}^{d-1}(\partial \Omega)$, we define a new set $\Gamma_1$ such that $\mathcal{H}^{d-1}(\Gamma_1)=m\mathcal{H}^{d-1}(\partial \Omega)$ and $\Gamma\subset \Gamma_1$, from Lemma \ref{lME} we know that $\lambda_k(\Omega,\Gamma_1)\leq \lambda_k(\Omega,\Gamma)$. This concludes the proof.
\end{proof}
\begin{remark}
We actually proved the existence of a minimizer in a relaxed framework. Indeed the minimizing set $\Gamma\subset \partial \Omega$ that we defined is not a relative open set, but we can still define the Steklov-Dirichlet eigenvalue $\lambda_k(\Omega,\Gamma)$. Indeed by construction we have that $\partial \Omega\setminus \Gamma=\cup_{i=1}^k\{x\in \partial \Omega \,|\,u_i(x)=0\}$, with $u_i\in H^1(\Omega)$ for all $i=1,...,k$, and in particular the space $H_0^1(\Omega,\Gamma)$ is a well defined Hilbert space. We can define the eigenvalues $\lambda_k(\Omega,\Gamma)$ via the variational characterization \eqref{eVC}. 
\end{remark}
We now study the maximization problem, in particular we prove Theorem \ref{tEM}. In order to prove that a maximizer does not exists it is enough to construct a sequence of domains $\Gamma_{S,n}\subset \partial \Omega$ such that $\mathcal{H}^{d-1}(\Gamma_{S,n})=m\mathcal{H}^{d-1}(\partial \Omega)$ for all $n$ and $\lambda_1(\Omega,\Gamma_{S,n})\rightarrow +\infty$. The following geometric Lemma is crucial for the construction of the maximizing sequence. This Lemma is classical (see \cite{HP18}), for this reason we will give only a sketch of the proof  
\begin{lemma}\label{lGL}
Let $\Omega$ be a Lipschitz domain and let $0<m<1$ be a constant, then there exists a sequence of domains $\Gamma_n\subset \partial \Omega$ such that 
\begin{align*}
\mathcal{H}^{d-1}(\Gamma_n)&=m\mathcal{H}^{d-1}(\partial \Omega)\quad \forall\, n \quad,\\ 
\chi_{\Gamma_n} & \overset{\ast}{\rightharpoonup}m\chi_{\partial \Omega} \quad \text{in} \quad L^{\infty}(\partial \Omega).\\ 
\end{align*}
\end{lemma}
\begin{proof}[Sketch of the Proof.] 
In order to give a more clear idea of the construction of the set $\Gamma_n$, we assume that there exists $c\in \mathbb{N}$ such that $m=\frac{1}{c}$.

Without loss of generality we can assume that $\mathcal{H}^{d-1}(\partial \Omega)=1$. Let $(\partial \Omega,g)$ be the manifold given by the boundary of $\Omega$ endowed with the metric $g$ induced by the euclidean metric on $\mathbb{R}^d$. We denote by $B_g(x,r)\subset \partial \Omega$ the ball with respect to the metric $g$, with center $x$ and radius $r$. We fix $n\in \mathbb{N}$ and we consider a set of points $\{x_k\}_{k=1}^{c^{n-1}}$ and a set of radii $\{r_{k,n}\}_{k=1}^{c^{n-1}}$ such that $x_j\notin B_g(x_k,r_{k,n})$ if $x_k\neq x_j$ and $\mathcal{H}^{d-1}(B_g(x_k,r_{k,n}))=\frac{1}{c^{n-1}}$ for all $k=1,2,...,c^{n-1}$. We define the following set
\begin{equation*}
\Gamma_n=\bigcup_{k=1}^{c^{n-1}}B_g(x_k,r_{k,n}),
\end{equation*}
now we have that $\mathcal{H}^{d-1}(\Gamma_n)=m$ for all $n$. Let $s\in L^1(\partial \Omega)$ be a step function, it is straightforward to check that $\int_{\partial
\Omega}\chi_{\Gamma_n}s\rightarrow m\int_{\partial
\Omega}s$ and we conclude by density (see \cite{HP18}). In the more general case where $c\notin \mathbb{N}$ we must define the radii $\{r_{k,n}\}_{k=1}^{c^{n-1}}$ in such a way that: 
\begin{equation*}
\sum_{k=1}^{\lfloor c^{n-1} \rfloor}\mathcal{H}^{d-1}(B_g(x_k,r_{k,n}))=c.
\end{equation*}
\end{proof}
We are now ready to prove Theorem \ref{tNEU}.
\begin{proof}[Proof of Theorem \ref{tNEU}]
Let $\Gamma_n$ be the sequence of subdomains defined in Lemma \ref{lGL} and we define $\Gamma_{S,n}=\Gamma_n$. Let $u_{1,n}$ be a first Steklov-Dirichlet eigenfunction associated to $\lambda_1(\Omega,\Gamma_{S,n})$ and we assume that $||u_{1,n}||_{L^2(\partial \Omega)}=1$.

Suppose by contradiction that there exists a constant $C$ such that:
\begin{equation*}
\sup_{n\in \mathbb{N}} \lambda_1(\Omega,\Gamma_{S,n})\leq C.
\end{equation*}
Now we know that $\lambda_1(\Omega,\Gamma_{S,n})=||\nabla u_{1,n}||_{L^2(\Omega)}^2\leq C$ and, from from Poincar\'e-Friedrichs inequality, we also know the following bound 
\begin{equation*}
\int_{\Omega}u_{1,n}^2dx\leq \frac{1}{\lambda_{1,1}(\Omega)}\big [\int_{\Omega}|\nabla u_{1,n}|^2 dx+\int_{\partial \Omega} u_{1,n}^2d\mathcal{H}^{d-1}\big ],
\end{equation*}
where $\lambda_{1,1}(\Omega)$ is the first Robin eigenvalue with parameter $1$. We conclude that the functions $u_{1,n}$ are bounded in $H^1(\Omega)$ and, from the fact that $\Omega$ is a Lipschitz domain, we can conclude also that there exists $u\in L^2(\partial \Omega)$ such that, up to a subsequence
\begin{equation*}
u_{1,n} \rightarrow u \quad \text{in} \quad L^2(\partial \Omega).
\end{equation*}
Now we reach a contradiction, indeed we know that $||u_{1,n}||_{L^2(\partial \Omega)}=1$ for all $n$, this implies
\begin{equation*}
1=\int_{\partial \Omega} u_{1,n}^2d\mathcal{H}^{d-1}\rightarrow \int_{\partial \Omega} u^2d\mathcal{H}^{d-1}=1,
\end{equation*}
but from Lemma \ref{lGL} we have that 
\begin{equation*}
1=\int_{\partial \Omega} u_{1,n}^2d\mathcal{H}^{d-1}=\int_{\partial \Omega}\chi_{\Gamma_{S,n}} u_{1,n}^2d\mathcal{H}^{d-1}\rightarrow m\int_{\partial \Omega} u^2d\mathcal{H}^{d-1}=m<1,
\end{equation*}
this is a contradiction.
\end{proof}
We want now to construct an explicit example in the plane where we can understand more deeply the divergence of the sequence $\lambda_1(\Omega,\Gamma_{S,n})$. This example actually show that the divergence of the sequence $\lambda_1(\Omega,\Gamma_{S,n})$ is linked to the unboundedness of the number of connected components of the sequence $\Gamma_{S,n}$. 

\begin{example}\label{ex2}
We consider the unit disk $\mathbb{D}$ in the plane and let $n\in \mathbb{N}$ be an even number. We define $\Gamma_{S,n}$ in the following way
\begin{equation*}
\Gamma_{S,n}=\bigcup_{k=0}^{\frac{n}{2}-1} \big \{e^{i\phi}\;|\;(2k)\frac{2\pi}{n}<\phi<(2k+1)\frac{2\pi}{n} \big \},
\end{equation*}

and, for all $k=0,...,n-1$, we define 

\begin{equation*}
\mathbb{D}_{k,n}=\big \{re^{i\phi}\;|\;0< r< 1,\;k\frac{2\pi}{n}+\frac{\pi}{n}< \phi< (k+1)\frac{2\pi}{n}+\frac{\pi}{n}\big \}.
\end{equation*}

\begin{figure}
\centering
%\tdplotsetmaincoords{0}{0}
\hspace*{5em}{\begin{tikzpicture}[scale=0.8] 
    %\node (Y) at (0,5) {$Im(z)$};
   
    %\draw[-stealth] (0,-5) -- (0,5) ;
    %\draw[-stealth] (-5,0) -- (5,0) ;
    %\draw [black,thick,domain=0:360] plot ({cos(\x)},{sin(\x)});
    \filldraw [color=black!60, fill=gray!5, thick] (0,0) -- (3.863703,1.035276) arc (15:45:4cm)
-- cycle;
\filldraw [color=black!60, fill=gray!5, thick] (0,0) -- (2.82842,2.82842) arc (45:75:4cm)
-- cycle;
\filldraw [color=black!60, fill=gray!5, thick] (0,0) -- (1.035276,3.863703) arc (75:105:4cm)
-- cycle;
    \draw[black] (4,0) arc (0:360:4);
    \draw[red, very thick] (4,0) arc (0:30:4cm);
    \draw[blue,very thick] (3.464101,2) arc (30:60:4cm);
    \draw[red,thick] (2,3.464101) arc (60:90:4cm);
    \draw[blue,very thick] (0,4) arc (90:120:4cm);
    \node (Y) at (4.4,1) {\textcolor{red}{$\Gamma_{S,n}$}};
    \node (X) at (3.5,3) {\textcolor{blue}{$\Gamma_{D,n}$}};
    \node (Z) at (1.5,4.15) {\textcolor{red}{$\Gamma_{S,n}$}};
    \node (W) at (-1,4.2) {\textcolor{blue}{$\Gamma_{D,n}$}};
    %\draw[thick, ->] (-0.5,1.5) arc (100:220:2);
    \draw [black,dashed,->,domain=120:360] plot ({cos(\x)},{sin(\x)});
    \node (Y1) at (2.5,1.3) {$\mathbb{D}_{i,n}$};
    \node (X1) at (1.5,2.5) {$\mathbb{D}_{i+1,n}$};
    \node (Z1) at (0,2.8) {$\mathbb{D}_{i+2,n}$};

  \end{tikzpicture}}
\caption{Example of the domain constructed in Example \ref{ex2}}
\label{fig1}
\end{figure}
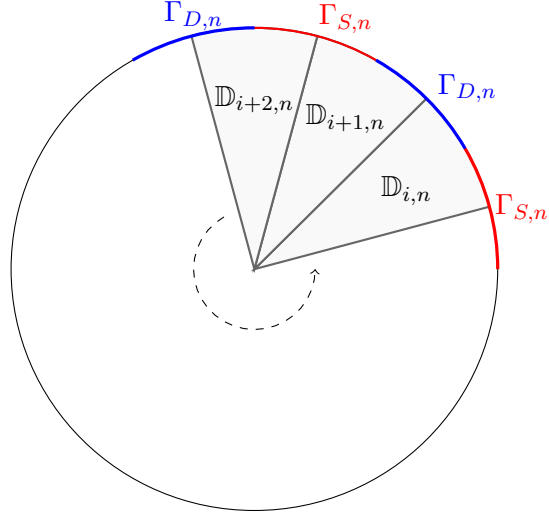

Let $u_{1,n}$ be the eigenfunction associated to $\lambda_1(\mathbb{D},\Gamma_{S,n})$, it is easy to check that there exists an index $\overline{j}$ such that the following holds
\begin{equation}\label{eLD}
\lambda_1(\mathbb{D},\Gamma_{S,n})=\frac{\int_{\mathbb{D}}|\nabla u_{1,n}|^2 dx}{\int_{\partial \mathbb{D}}u_{1,n}^2 ds}=\frac{\sum_{k=0}^{n-1}\int_{\mathbb{D}_{k,n}}|\nabla u_{1,n}|^2 dx}{\sum_{k=0}^{n-1}\int_{\partial \mathbb{D}\cap \partial \mathbb{D}_{k,n} }u_{1,n}^2 ds}\geq \frac{\int_{\mathbb{D}_{\overline{j},n}}|\nabla u_{1,n}|^2 dx}{\int_{\partial \mathbb{D}\cap \partial \mathbb{D}_{\overline{j},n} }u_{1,n}^2 ds}.
\end{equation}
Now by induction we define a new function $\overline{u}_n$, in $\mathbb{D}$, we define $\overline{u}_n$ to be equal to $u_{1,n}$ on $\mathbb{D}_{\overline{j},n}$ , on $\mathbb{D}_{\overline{j}+1,n}$ we define $\overline{u}_n$ in the following way
\begin{equation*}
\overline{u}_n(r,\phi)=\overline{u}_n\big (r,(\overline{j}+1)\frac{4\pi}{n}+\frac{2\pi}{n}-\phi \big )\quad \forall (r,\phi)\in \mathbb{D}_{\overline{j}+1,n},
\end{equation*}
and, for all $k$, we will define $\overline{u}_n$ on $\mathbb{D}_{k+1,n}$ knowing the function on $\mathbb{D}_{k,n}$ in the following way
\begin{equation*}
\overline{u}_n(r,\phi)=\overline{u}_n\big (r,(k+1)\frac{4\pi}{n}+\frac{2\pi}{n}-\phi \big )\quad \forall (r,\phi)\in \mathbb{D}_{k+1,n}.
\end{equation*}

From the definition we know that $\overline{u}_n\in H^1(\Omega)$, for all $0<r_0\leq 1$ the function $\overline{u}_n(r_0,\phi)$ is periodic in $[0,2\pi]$ with period $\frac{2\pi}{n}$ and also, from \eqref{eLD}, we know that
\begin{equation}\label{eLD2}
\lambda_1(\mathbb{D},\Gamma_{S,n})\geq \frac{\int_{\mathbb{D}_{\overline{j},n}}|\nabla u_{1,n}|^2 dx}{\int_{\partial \mathbb{D}\cap \partial \mathbb{D}_{\overline{j},n} }u_{1,n}^2 ds}=\frac{\int_{\mathbb{D}_{\overline{j},n}}|\nabla \overline{u}_n|^2 dx}{\int_{\partial \mathbb{D}\cap \partial \mathbb{D}_{\overline{j},n} }\overline{u}_n^2 ds}=\frac{\int_{\mathbb{D}}|\nabla \overline{u}_n|^2 dx}{\int_{\partial \mathbb{D}}\overline{u}_n^2 ds}.
\end{equation}
The inequality above prove that the function $\overline{u}_n$ is an eigenfunction associated to $\lambda_1(\mathbb{D},\Gamma_{S,n})$, in particular is an harmonic function on $\mathbb{D}$. 

Finally we know that $\overline{u}_n$ is harmonic in $\mathbb{D}$ and $\overline{u}_n|_{\partial \mathbb{D}}$ is periodic in $[0,2\pi]$ with period $\frac{2\pi}{n}$, so let $c_l$ be the Fourier coefficients for the function $\overline{u}_n|_{\partial \mathbb{D}}$, the following equalities holds
\begin{align*}
\int_{\partial \mathbb{D}}\overline{u}_n^2 ds&=\sum_{l=1}^{\infty}c_l^2,\\
\int_{\mathbb{D}}|\nabla \overline{u}_n|^2 dx&=\pi n \sum_{l=1}^{\infty}lc_l^2.
\end{align*}  
From the inequalities above and \eqref{eLD2} we finally conclude that:
\begin{equation*}
\lambda_1(\mathbb{D},\Gamma_{S,n})=\frac{\int_{\mathbb{D}}|\nabla \overline{u}_n|^2 dx}{\int_{\partial \mathbb{D}}\overline{u}_n^2 ds}\geq \pi n
\end{equation*}
so we finally have that $\lim_{n\rightarrow \infty}\lambda_1(\mathbb{D},\Gamma_{S,n})=+\infty$.

Now let $\Omega$ be a simply connected domain and let $f:\Omega\rightarrow \mathbb{D}$ be a conformal map. Suppose that $\Omega$ is regular enough so that $\min_{z\in \partial \Omega}|f'(z)|>0$, for instance $\Omega$ is a inner domain of a Dini-Jordan curve (see \cite{P92}). We define $\hat{\Gamma}_{S,n}:=f^{-1}(\Gamma_{S,n})$, let $v_{1,n}$ be the eigenfunction associated to $\lambda_1(\Omega,\hat{\Gamma}_{S,n})$ and let $\hat{v}_n=v_{1,n}\circ f^{-1}$, using the test function $\hat{v}_n$ in the variational characterization of $\lambda_1(\mathbb{D},\Gamma_{S,n})$ we obtain that
\begin{equation*}
\lambda_1(\mathbb{D},\Gamma_{S,n})\leq \frac{\int_{\mathbb{D}}|\nabla \hat{v}_n|^2 dx}{\int_{\partial \mathbb{D}}\hat{v}_n^2 ds}=\frac{\int_{\Omega}|\nabla v_{1,n}|^2 dx}{\int_{\partial \Omega}v_{1,n}^2|f'| ds}.
\end{equation*}
So we have that
\begin{equation*}
\min_{z\in \partial \Omega}|f'(z)|\lambda_1(\mathbb{D},\Gamma_{S,n})\leq \frac{\int_{\Omega}|\nabla v_{1,n}|^2 dx}{\int_{\partial \Omega}v_{1,n}^2 ds}=\lambda_1(\Omega,\hat{\Gamma}_{S,n})
\end{equation*}
and finally we obtain that $\lim_{n\rightarrow \infty}\lambda_1(\Omega,\hat{\Gamma}_{S,n})=+\infty$.
\end{example}
\subsection{Continuity of eigenvalues under topological constraints}
Motivated by the Example \ref{ex2}, and also by the explicit construction that we use to prove Theorem \ref{tEM},  it is natural to ask what happened if we put some topological constraint that prevent the phenomenon of the diffusion of $\Gamma_{S,n}$ over all the boundary. In dimension $d=2$ it is natural to put the constraint on the maximal number of connected components. 

With this constraint in dimension $d=2$ we can compare the $L^1$ distance of two sets in $\partial \Omega$ with the Hausdorff distance. This comparison, in general without any topological constraint, is not possible (see \cite{HP18}). Due to this comparison, that is now possible thanks to the topological constraint, we also have a nice control on the costant $C_2$ in Theorem \ref{tSL} and we can prove a continuity result for the Steklov-Dirichlet eigenvalues. 

A similar continuity result for Dirichlet eigenvalues was proved by V. \v{S}ver\'ak in \cite{S93}. More precisely he proved the following result: let $\lambda_k(\Omega)$ the $k-$th Dirichlet eigenvalue of the set $\Omega$ and we define the following 
\begin{equation*}
\#(\Omega)=\text{number of connected components of }\Omega
\end{equation*}
\begin{theorem}[\textbf{V. \v{S}ver\'ak} \cite{HP18, S93}]\label{tS}
Let $\Omega_n\subset \mathbb{R}^2$ be a sequence of bounded open sets converging for
the Hausdorff metric to an open set $\Omega$. Assume that there exists a constant $C$ such that $\#(\Omega_n)\leq C$ for all $n$, then, for all $k$, the Dirichlet eigenvalues converge
\begin{equation*}
\lambda_k(\Omega_n)\rightarrow \lambda_k(\Omega).
\end{equation*}
\end{theorem}
We prove now Theorem \ref{tSS} that is the equivalent of Theorem \ref{tS} in the context of Steklov-Dirichlet eigenvalues.

\begin{proof}[Proof of Theorem \ref{tSS}]
Let $M$ be a constant such that $\#(\Gamma_{D,n})\leq M$ for all $n$, let $(\partial \Omega,s)$ be the $1-$dimensional manifold endowed with the length distance $s$ and let $d_{(\partial \Omega,s)}$ be the Hausdorff distance with respect to the length distance. From the boundedness of the connected components and from the regularity assumption on $\Omega$ (the boundary is parametrized by smooth functions with uniform bounded derivatives, see Definition \ref{dOm}) we conclude that there exists a constant $C_1(M)$ that depends only on $M$ and a constant $C_2(\Omega)$ that depends only on $\Omega$, such that 
\begin{equation*}
\mathcal{H}^1(\Gamma_{D,n}\triangle \Gamma_D )\leq C_1(M)d_{(\partial \Omega,s)}(\Gamma_{D,n},\Gamma_D)\leq C_1(M)C_2(\Omega)d(\Gamma_{D,n},\Gamma_D)
\end{equation*}
and in particular 
\begin{align}\label{eHMSD}
\mathcal{H}^1(\Gamma_{D,n}\triangle \Gamma_D )&\rightarrow 0,\\ \notag
\mathcal{H}^1(\Gamma_{S,n}\triangle \Gamma_S )&\rightarrow 0.
\end{align}
From Theorem \ref{tSL} we have that there exists a sequence of constants $C_{2,n}$ such that  
\begin{equation}\label{eCCC}
|\lambda_1(\Omega,\Gamma_{S,n})-\lambda_1(\Omega,\Gamma_S)|\leq C_{2,n}\big (\mathcal{H}^1(\Gamma_{S,n} \triangle \Gamma_S)^{\frac{1}{2}}+d(\Gamma_{D,n},\Gamma_D)^{\frac{1}{2}}\big ).
\end{equation}
Now we prove that there exists a constant $L$ such that $C_{2,n}\leq L$ for all $n$. We call $v_{k,n}$ the function constructed in Theorem \ref{tSL} when we consider $\Gamma'_S=\Gamma_{S,n}$. From the proof of Theorem \ref{tSL} it is clear that $C_{2,n}\leq L$ for all $n$ if and only if there exist a constant $C>0$ such that, the three following estimates hold
\begin{equation*}
\int_{\partial \Omega} v_{k,n}^2 d\mathcal{H}^1>\frac{1}{C},
\end{equation*} 
\begin{equation*}
||v_{k,n}||_{H^1(\Omega)}\leq C
\end{equation*} 
and 
\begin{equation*}
\max \{\lambda_k(\Omega,\Gamma_{S,n}),\lambda_k(\Omega,\Gamma_S) \}\leq C.
\end{equation*} 
Suppose by contradiction that 
\begin{equation}\label{eCVK}
\liminf_{n\to \infty}\int_{\partial \Omega} v_{k,n}^2 d\mathcal{H}^1=0.
\end{equation}
Let $u_k$ be the Steklov-Dirichlet eigenfunction associated to $\lambda_k(\Omega,\Gamma_S)$. From Theorem 4 in \cite{S97}, and using similar arguments as in \eqref{eUBHN2}, we have that:
\begin{equation*}
||u_k-v_{k,n}||_{H^1(\Omega)}\leq C\big (\mathcal{H}^1(\Gamma_{S,n} \triangle \Gamma_S)^{\frac{1}{2}}+d(\Gamma_{D,n},\Gamma_D)^{\frac{1}{2}}\big ),
\end{equation*}
in particular 
\begin{align*}
v_{k,n} &\rightarrow u_k \quad \text{in} \quad H^1(\Omega),\\ \notag
v_{k,n} &\rightarrow u_k \quad \text{in} \quad L^2(\partial \Omega).
\end{align*}
From \eqref{eCVK} we conclude that $\int_{\partial \Omega} u_k^2 d\mathcal{H}^1=0$ that is a contradiction, because $u_k$ is a Steklov-Dirichlet eigenfunction. From the convergence above it is also clear that $||v_{k,n}||_{H^1(\Omega)}$ is bounded.

Now we prove an upper bound for the quantity $\max \{\lambda_k(\Omega,\Gamma_{S,n}),\lambda_k(\Omega,\Gamma_S) \}$. From \eqref{eHMSD} we know that there exist a set $\Gamma$ such that $\mathcal{H}^1(\Gamma)>0$ and a natural number $N$ such that 
\begin{equation*}
\Gamma\subset \Gamma_{S,n}\cap \Gamma_S\quad \forall\, n\geq N.
\end{equation*}
Now by Lemma \ref{lME} we conclude that:
\begin{equation*}
\max \{\lambda_k(\Omega,\Gamma_{S,n}),\lambda_k(\Omega,\Gamma_S) \}\leq \lambda_k(\Omega,\Gamma)\quad \forall\, n\geq N.
\end{equation*}
We finally proved that $C_{2,n}\leq L$ for all $n$. From the boundedness of $C_{2,n}$, from \eqref{eCCC}, from \eqref{eHMSD} and from the assumption that $\Gamma_{D,s}$ Hausdorff converge to $\Gamma_D$ we conclude that 
\begin{equation*}
\limsup_{n\to \infty} |\lambda_1(\Omega,\Gamma_{S,n})-\lambda_1(\Omega,\hat{\Gamma}_S)|\leq \limsup_{n\to \infty} C_{2,n}\big (\mathcal{H}^1(\Gamma_{S,n} \triangle \hat{\Gamma}_S)^{\frac{1}{2}}+d(\Gamma_{D,n},\hat{\Gamma}_D)^{\frac{1}{2}}\big )=0.
\end{equation*}
This concludes the proof
\end{proof}
Using this continuity result we can now prove the following existence Theorem:
\begin{theorem}\label{tEMCC}
Let $\Omega\subset \mathbb{R}^2$ be a $\mathcal{C}^{1,1}$ open domain, let $0<m_1<1$ and $m_2\in\mathbb{N}$ be two constants, then the following problem
\begin{equation*}
\sup \{ \lambda_k(\Omega,\Gamma_S) \;| \; \mathcal{H}^1(\Gamma_S)=m_1\mathcal{H}^1(\partial \Omega)\; \text{and}\; \#(\Gamma_S)\leq m_2 \}
\end{equation*}
has a solution.
\end{theorem}

\begin{proof}
We define $\mathcal{A}=\{\Gamma_S\subset \partial\Omega | \; \mathcal{H}^1(\Gamma_S)=m_1\mathcal{H}^1(\partial \Omega)\; \text{and}\; \#(\Gamma_S)\leq m_2 \}$ and we consider a maximizing sequence $\Gamma_{S,n}$, we define also the following compact sets $\Gamma_{D,n}=\partial \Omega\setminus \Gamma_{S,n}$. For all $n$ we know that $\Gamma_{S,n}\in \mathcal{A}$, so from the constraint on the measure and the constraint on the number of connected components we have that
\begin{equation*}
||\chi_{\Gamma_{S,n}}||_{BV(\partial \Omega)}= \mathcal{H}^1(\Gamma_{S,n})+\mathcal{H}^0(\Gamma_{S,n})\leq m_1\mathcal{H}^1(\partial \Omega)+2m_2\quad \forall\, n. 
\end{equation*}
From the compactness property of the space $BV(\partial \Omega)$ we have that there exists a set $\hat{\Gamma}_S\subset \partial \Omega$ such that, up to a subsequence
\begin{equation*}
\chi_{\Gamma_{S,n}}\rightarrow \chi_{\hat{\Gamma}_S}\quad \text{in} \quad L^1(\partial \Omega).
\end{equation*}
We define the set $\hat{\Gamma}_D=\partial \Omega\setminus \hat{\Gamma}_D$, from the convergence above we conclude that: 
\begin{align}\label{eHMSD2}
\mathcal{H}^1(\Gamma_{S,n}\triangle \hat{\Gamma}_S )&\rightarrow 0,\\ \notag
\mathcal{H}^1(\Gamma_{D,n}\triangle \hat{\Gamma}_D )&\rightarrow 0.
\end{align}
Let $s$ be the length distance on $\partial \Omega$,  for all $x_1\in \partial \Omega $ and $x_2\in \partial \Omega $ we have that $|x_1-x_2|\leq s(x_1,x_2)$ and in particular
\begin{equation}\label{eHDHDM2}
d(\Gamma_{D,n},\hat{\Gamma}_D)\leq d_{(\partial \Omega,s)}(\Gamma_{D,n},\hat{\Gamma}_D ).
\end{equation}
Where we denoted by $d_{(\partial \Omega,s)}(\Gamma_1,\Gamma_2)$ the Hausdorff distance with respect the length distance on $\partial \Omega$. The boundary $\partial \Omega$ is a $1-$dimensional manifold, so it is clear that
\begin{equation}\label{eHDHM2}
d_{(\partial \Omega,s)}(\Gamma_{D,n},\hat{\Gamma}_D)\leq \mathcal{H}^1(\Gamma_{D,n}\triangle \hat{\Gamma}_D ),
\end{equation}
so from \eqref{eHMSD2} we conclude that $\Gamma_{D,s}$ Hausdorff converge to $\hat{\Gamma}_D$. From Theorem \ref{tSS} we conclude that 
\begin{equation*}
\lambda_k(\Omega,\Gamma_{S,n})\rightarrow \lambda_k(\Omega,\hat{\Gamma}_S).
\end{equation*}
The fact that $\hat{\Gamma}_S\in \mathcal{A}$ is straightforward, the measure constraint comes directly from \eqref{eHMSD2} and the bounds on the number of connected components is preserved by the Hausdorff distance (see \cite{HP18}).
\end{proof}

\bigskip\noindent

{\bf Acknowledgements}: 
The author is grateful to A. Henrot for proposing the problem, for fruitful discussions and for remarks on the preliminary version of the manuscript. The author is also grateful to D. Bucur for fruitful discussions about Theorem \ref{tEM}. This work was supported by the project ANR-18-CE40-0013 SHAPO financed by the French Agence Nationale de la Recherche (ANR).

\bibliographystyle{abbrv}
\bibliography{Ref}

\end{document}